\documentclass[11pt]{amsart}
\usepackage[foot]{amsaddr}
\usepackage{amsmath,amssymb,bm, graphicx,float}

\begin{document}
\newtheorem{theorem}{Theorem}[section]
\newtheorem{lemma}[theorem]{Lemma}
\newtheorem{corollary}[theorem]{Corollary}
\newtheorem{prop}[theorem]{Proposition}
\newtheorem{definition}[theorem]{Definition}
\newtheorem{remark}[theorem]{Remark}

 \def\ad#1{\begin{aligned}#1\end{aligned}}  \def\b#1{{\bf #1}} \def\hb#1{\hat{\bf #1}}
\def\a#1{\begin{align*}#1\end{align*}} \def\an#1{\begin{align}#1\end{align}}
\def\e#1{\begin{equation}#1\end{equation}} \def\t#1{\hbox{\rm{#1}}}
\def\dt#1{\left|\begin{matrix}#1\end{matrix}\right|}
\def\p#1{\begin{pmatrix}#1\end{pmatrix}} \def\c{\operatorname{curl}}
 \def\vc{\nabla\times } \numberwithin{equation}{section}
 \def\la{\circle*{0.25}}
\def\boxit#1{\vbox{\hrule height1pt \hbox{\vrule width1pt\kern1pt
     #1\kern1pt\vrule width1pt}\hrule height1pt }}
 \def\lab#1{\boxit{\small #1}\label{#1}}
  \def\mref#1{\boxit{\small #1}\ref{#1}}
 \def\meqref#1{\boxit{\small #1}\eqref{#1}}
\long\def\comment#1{}

\def\lab#1{\label{#1}} \def\mref#1{\ref{#1}} \def\meqref#1{\eqref{#1}}

\def\bg#1{{\pmb #1}} 

\title  [rectangular Bell elements]
   {Rectangular $C^1$-$Q_k$  Bell finite elements in two and three dimensions}

\author {Hongling Hu}
\address{School of Mathematics and Statistics, MOE-LCSM, Hunan Normal University, Changsha 410081, China} 
\email{honglinghu@hunnu.edu.cn}
\thanks{
Hongling Hu is supported by the National Natural Science Foundation
of China (No. 12071128) and the Natural Science Foundation of Hunan Province,
China (No. 2021JJ30434). }

\author { Shangyou Zhang }
\address{Department of Mathematical  Sciences, University of Delaware, Newark, DE 19716, USA.}
\email{ szhang@udel.edu }

\date{}

\begin{abstract}
Both the function and its normal derivative on the element boundary are $Q_k$ polynomials
  for the Bogner-Fox-Schmit $C^1$-$Q_k$ finite element functions.
Mathematically, to keep the optimal order of approximation,  their spaces are required to
   include $P_k$ and $P_{k-1}$ polynomials respectively.
We construct a Bell type $C^1$-$Q_k$ finite element on rectangular meshes in 2D and 3D,
   which has its normal derivative as a $Q_{k-1}$ polynomial on each face, for $k\ge 4$.
We show, with a big reduction of the space, the $C^1$-$Q_k$ Bell
  finite element retains the optimal order of convergence.
Numerical experiments are performed, comparing the new elements with the original elements.
  
\end{abstract}

\vskip .3cm



\maketitle

\section{Introduction} 
The construction and analysis on the finite elements mainly started 
   when some engineers use the method to solve the following plate bending equation,
\an{\label{bi} \ad{ \Delta^2 u & = f \quad \t{in } \ \Omega, \\
        u=\partial_{\b n} u & =0  \quad \t{on } \ \partial\Omega, } }
where $\Omega$ is a polygonal domain in 2D or 3D, and $\b n$ is a normal vector,
cf.  the $C^1$-$P_3$ Hsieh-Clough-Tocher element (1961,1965) \cite{Ciarlet,Clough}, 
  the $C^1$-$P_3$ Fraeijs de Veubeke-Sander element (1964,1965) \cite{Fraeijs,Fraeijs68,Sander}  
 the $C^1$-$P_5$  Argyris element (1968) \cite{Argyris},
  the $C^1$-$P_4$ Bell element (1969) \cite{Bell},
 and  the $C^1$-$Q_3$ Bogner-Fox-Schmit (BFS)  element (1965) \cite{Bogner}.
 
These lowest order elements are extended to general polynomial degrees $k$.
The $C^1$-$P_3$ Hsieh-Clough-Tocher element was extended to the $C^1$-$P_k$ ($k\ge 3$) 
   finite elements in \cite{Douglas,ZhangMG}. 
The $C^1$-$Q_3$ Bogner-Fox-Schmit element was extended to three families of $C^1$-$Q_k$ ($k\ge 3$)
  finite elements on rectangular meshes in \cite{Zhang-C1Q}.
The $C^1$-$P_3$ Fraeijs de Veubeke-Sander element is extended to two families
  of $C^1$-$P_k$ ($k\ge 3$) finite elements in \cite{Zhang-F}.  
The $C^1$-$P_5$  Argyris element was extended to the family of 
   $C^1$-$P_k$ ($k\ge 5$) finite elements in \cite{Zen70,Zlamal}. 
The $C^1$-$P_5$  Argyris element was modified and extended to the family of  
   $C^1$-$P_k$ ($k\ge 5$) full-space finite elements in \cite{Morgan-Scott}.
The $C^1$-$P_5$  Argyris element was also extended to 3D $C^1$-$P_k$ ($k\ge 9$)
   elements on tetrahedral meshes in \cite{Zenisek,Z3d,Z4d}.
The $C^1$-$P_4$ Bell element was extended to three families of 
    $C^1$-$P_{2m+1}$ ($m\ge 3$) finite elements in \cite{Xu-Zhang7,Xu-Zhang}. 
The Bell finite elements do not have any degrees of freedom on edges. 
Thus they must be odd-degree polynomials 
  (the $P_4$ Bell element is a subspace of $P_5$ polynomials.)
The $C^1$-$P_k$ Bell element space is a reduced space, i.e., a subspace of the space of the
    $C^1$-$P_{k}$ Argyris finite element. 
  
In this work, we consider such a reduced space for the $C^1$-$Q_k$ Bogner-Fox-Schmit element.
Both the function and its normal derivative on the element boundary are $Q_k$ polynomials
  for the Bogner-Fox-Schmit $C^1$-$Q_k$ finite element functions.
Mathematically, to keep the optimal order of approximation,  their spaces are required to
   include $P_k$ and $P_{k-1}$ polynomials respectively.
Thus, we construct a Bell type $C^1$-$Q_k$ finite element on rectangular meshes in 2D and 3D,
   which has its normal derivative as a $Q_{k-1}$ 
     (instead of $Q_k$) polynomial on each face, for $k\ge 4$, cf. Figure \ref{f-2e}.

\begin{figure}[H] \centering \setlength\unitlength{1.8pt}
\begin{picture}(180,90)(0,0)
\def\gds{\begin{picture}(160,160)(0,0)
  \multiput(0,0)(40,0){2}{\line(0,1){40}}\multiput(0,0)(0,40){2}{\line(1,0){40}}
  \multiput(0,0)(20,0){3}{\multiput(0,0)( 0,20){3}{\circle*{2}}}
  \multiput(1,1)(40,0){2}{\multiput(0,0)( 0,40){2}{\vector(1,0){6}}}
  \multiput(1,1)(40,0){2}{\multiput(0,0)( 0,40){2}{\vector(0,1){6}}}
  \multiput(1,0)(40,0){2}{\multiput(0,0)( 0,40){2}{\vector(1,1){6}}}
  \multiput(0,1)(40,0){2}{\multiput(0,0)( 0,40){2}{\vector(1,1){6}}}\end{picture} }
\put(0,0){\begin{picture}(80,80)(0,0)
  \multiput(0,0)(0,40){2}{\multiput(0,0)(40,0){2}{\gds}}
  \multiput(20,1)(40,0){2}{\multiput(0,0)( 0,40){3}{\vector(0,1){6}}}
  \multiput(1,20)(0,40){2}{\multiput(0,0)(40, 0){3}{\vector(1,0){6}}}
  \end{picture} }
  
\put(100,0){\begin{picture}(80,80)(0,0)
  \multiput(0,0)(0,40){2}{\multiput(0,0)(40,0){2}{\gds}}
  \end{picture} }
   \end{picture}
   \caption{\label{f-2e}Left: degrees of freedom of the $C^1$-$Q_4$ BFS finite element on $n^2$ rectangles:
        $9n^2+12n+4$; Right: degrees of freedom of the $C^1$-$Q_4$ Bell finite element: $ 7n^2+10n+4$.
         }
   \end{figure}
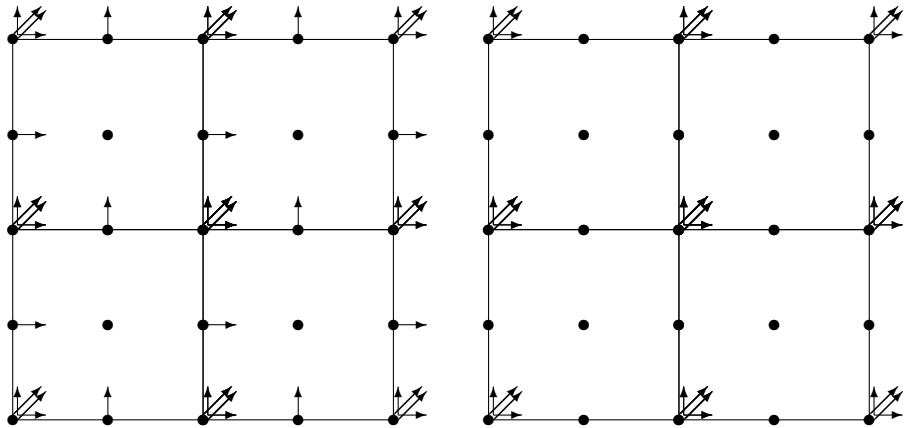
  
The idea is from the Bell triangular element, restricting the normal derivative to the space of
  one degree less polynomials.
But the idea is better used here as the order of convergence of the original $C^1$-$Q_k$ element
   is retained, while it is reduced by one order in the case of triangular Bell element.
We can see the reduction from Figure \ref{f-2e} where the lowest order 2D case is depicted.
On one element, the number of degrees of freedom is reduced from $25$ to $21$, only the
  normal derivative dof is removed at each of the four mid-edge points.
But globally, the dof ratio is 9:7, or equivalently the original method 
   has a $28\%$ more number of unknowns.
The reduction in 3D is even more. To be exact, the number of global degrees of freedom of
  the Bogner-Fox-Schmit $C^1$-$Q_4$ element in 3D is $50\%$ more than that of its Bell element,
    cf. Table \ref{t5}.

In this work, we show the uni-solvency and the optimal order of convergence
   of the new Bell $C^1$-$Q_k$ elements ($k\ge 4$) in 2D and 3D.
The theory is confirmed by numerical tests.
 
\section{The 2D $C^1$-$Q_k$ Bell finite elements}

Let $\mathcal Q_h=\{T\}$ be a square-of-size-$h$ mesh in 2D,
   or a cube-of-size-$h$ mesh in 3D, on the polygonal or polyhedral
  domain $\Omega$ in \eqref{bi}.
  
On a square or a cube,  we define, for $k\ge 4$,  
\an{\label{V-T} V_T =\{ v \in Q_k(T) : \partial_{\b n} v|_e \in Q_{k-1}(e), \ e \in \partial T\}, }
where $\partial_{\b n}$ denotes a normal derivative on the edge or the face-square $e$,
   and $Q_k=\t{span} \{ x_1^{ k_1}\cdots x_d^{k_d} : 0\le k_i\le k \}$ for $d=1,2,3$. 

For 2-dimensional $V_T$ on a size $h$ square $T$,  
    the degrees of freedom of the Bell element are defined by, cf.
   Figure \ref{f-dof2}, 
\an{\label{dof2} F_m(v) = \begin{cases}
     v ,  & \t{at } \ \b x_1+\frac h{k-2}(i, j), \; i,j=0,\dots,k-2, \\
     \partial_x v, 
          & \t{at } \ \b x_1+ h (i,\frac j{k-3}), \; i=0,1, \ j=0,\dots,k-3, \\ 
     \partial_y v, 
          & \t{at } \ \b x_1+ h (\frac i{k-3},j) , \; i=0,\dots, k-3,\ j=0,1, \\
     \partial_{xy} v, 
          & \t{at } \ \b x_1+h(i,j), \; i,j=0, 1. \end{cases} }

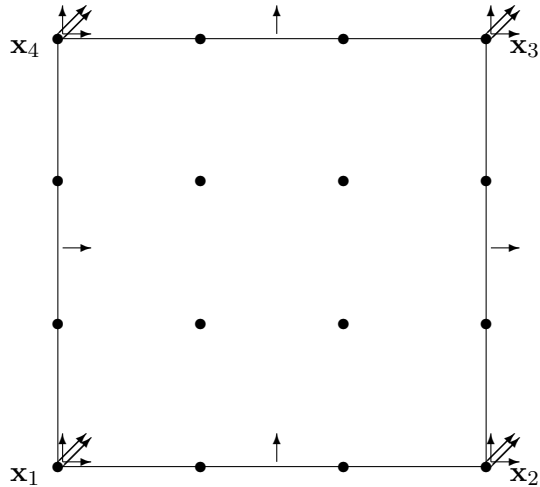
\begin{figure}[H] \centering \setlength\unitlength{1.8pt}
\begin{picture}(110,110)(0,0)
\put(10,5){\begin{picture}(110,110)(0,0) 
  \put(-10,-3){$\b x_1$}   \put(95,-3){$\b x_2$} 
  \put(-10,87){$\b x_4$}   \put(95,87){$\b x_3$} 
  \multiput(0,0)(90,0){2}{\line(0,1){90}}\multiput(0,0)(0,90){2}{\line(1,0){90}}
  \multiput(0,0)(30,0){4}{\multiput(0,0)( 0,30){4}{\circle*{2}}}
  \multiput(1,1)(90,0){2}{\multiput(0,0)( 0,45){3}{\vector(1,0){6}}}
  \multiput(1,1)(0,90){2}{\multiput(0,0)(45, 0){3}{\vector(0,1){6}}}
  \multiput(1,0)(90,0){2}{\multiput(0,0)( 0,90){2}{\vector(1,1){6}}}
  \multiput(0,1)(90,0){2}{\multiput(0,0)( 0,90){2}{\vector(1,1){6}}}\end{picture} } 
   \end{picture}
   \caption{\label{f-dof2} The degrees of freedom of the $C^1$-$Q_5$ Bell finite element,
       cf. \eqref{dof2}.
         }
   \end{figure}

\begin{lemma}  \label{l2d}
The set of degrees of freedom in \eqref{dof2} uniquely define a $V_T$ function in 
  \eqref{V-T}.
\end{lemma}

\begin{proof} The number of degrees of freedom in \eqref{dof2} is
\an{ \label{n2dof} N_d &=(k-1)^2 + 2(2(k-2))+ 4. }
The number of constraint equations in \eqref{V-T} is, one derivative restriction each
   edge,
\an{\label{n2e} N_e &=4. }
Adding \eqref{n2dof} and \eqref{n2e},  we get
\a{ N_d+N_e &= (k-1)^2 + 4(k-1) + 4=(k+1)^2=\dim Q_k.  }
Thus, to determine a $v_h\in V_T$,  we have a square system of linear equations of $(k+1)^2$
   unknowns.
The uniqueness implies the existence of the solution.

Let $v_h\in V_T$ and $F_m(v_h)=0$ for all dofs in \eqref{dof2}.   We need to show $v_h=0$.

Because $v$ vanishes at $k-1$ points on an edge $\b x_1\b x_2$ and $\partial_x v$ vanishes at
  the two end points, the degree-$k$ polynomial $v|_{\b x_1\b x_2}=0$.
Which implies
\a{ v_h = (y-(\b x_1)_2) q_{k,k-1}}
for some $\ q_{k,k-1}\in \t{span} \{ x^i y^j,
   \ 0\le i\le k, \ 0\le j\le k-1\}$, 
   where $(\b x_1)_2$ is the $y$-coordinate of $\b x_1$.

Because the degree $k-1$ polynomial $\partial_y v|_{\b x_1\b x_2}$ vanishes at $k-2$ points on 
   $\b x_1\b x_2$ and its tangential derivative ($\partial_{xy}v_h$) vanishes at
  the two end points, we  $\partial_y v|_{\b x_1\b x_2}=0$.
Thus, we can factor out another factor that 
\a{ v_h = (y-(\b x_1)_2)^2 q_{k,k-2}}
for some $\ q_{k,k-2}\in \t{span} \{ x^i y^j,
   \ 0\le i\le k, \ 0\le j\le k-2\}$.
  
  Repeating the calculation on the four edges, we have 
\a{v_h = (x-(\b x_1)_1)^2(x-(\b x_3)_1)^2 (y-(\b x_1)_2)^2(y-(\b x_3)_2)^2q_{k-4}}
for some $\ q_{k-4 }\in Q_{k-4}$.  Evaluating $v_h$ at the inner $(k-3)^2$
   points inside $T$, cf. Figure \ref{f-dof2},  we conclude that $q_{k-4}=0$ and $v_h=0$.
  The lemma is proved. 
\end{proof}

\section{The 3D $C^1$-$Q_k$ Bell finite elements}
 
For the 3-dimensional $V_T$ in \eqref{V-T}, on a size $h$ cube $T$,  
    the degrees of freedom of the Bell element are defined by, cf.
   Figure \ref{f-dof3}, $F_m(v) = $ 
\an{\label{dof3} \begin{cases}
     v ,\partial_x v, \partial_y v,   \partial_z v,  \partial_{xy} v,  \partial_{xz} v, 
      \partial_{yz} v, \partial_{xyz} v, 
         \ \quad \t{at eight vertices}, \\
      v,  \ \quad\qquad\qquad \qquad\qquad\t{at $k-3$ mid-points on 12 edges}, \\
    \partial_{\b n_1} v, \partial_{\b n_2} v, \partial_{\b n_1\b n_2} v,\ \quad 
        \t{at $k-4$ mid-points on 12 edges}, \\
      v,  \ \quad \qquad\qquad \qquad \t{at $(k-3)^2$ mid-points on 6 face-squares}, \\
  \partial_{\b n} v,  \ \qquad \qquad\qquad \t{at $(k-4)^2$ mid-points on 6 face-squares}, \\
      v,  \ \qquad\quad \qquad \qquad \t{at $(k-3)^3$ mid-points inside the cube}. \end{cases} }
     
\begin{figure}[H] \centering \setlength\unitlength{2.2pt}
\begin{picture}(70,70)(0,0)
\put(5,5){\begin{picture}(110,110)(0,0) 
  \put(-8,-2){$\b x_2$} \put(43,-2){$\b x_3$} \put(63,18){$\b x_4$} \put(63,58){$\b x_8$} 
  \put(-10,38){$\b x_6$}  
  \multiput(0,0)(40,0){2}{\line(0,1){40}}\multiput(0,0)(0,40){2}{\line(1,0){40}}
  \multiput(40,0)(20,20){2}{\line(0,1){40}}\multiput(0,40)(20,20 ){2}{\line(1,0){40}}
  \multiput(40,0)( 0,40){2}{\line(1,1){20}}\multiput(0,40)(20,20){1}{\line(1,1){20}}
 %  \multiput(0,0)(13.33,0){4}{ \multiput(0,0)(0,13.33 ){4}{\circle*{2}}}
  \multiput(0,40)(13.33,0){4}{ \multiput(0,0)(6.66,6.66){4}{\circle*{2}}}
 %  \multiput(40, 0)(0,13.33){4}{ \multiput(0,0)(6.66,6.66){4}{\circle*{2}}}
  \multiput(0.5,40)(20,0){3}{ \multiput(0,0)(10,10){3}{\vector(0,1){6}}}
  \multiput(40,40.3)(40,0){1}{ \multiput(0,0)(20,20){2}{\vector(1, -1){5}}}
  \multiput(40.3,40.6)(40,0){1}{ \multiput(0,0)(20,20){2}{\vector(1, -1){5}}} 
  
  \multiput(-0.6,40.3)(40,0){2}{ \multiput(0,0)(20,20){2}{\vector(-1, 2){3}}}
  \multiput(-0.7,39.6)(40,0){2}{ \multiput(0,0)(20,20){2}{\vector(-1, 2){3}}}
  \multiput( 0.6,40.3)(40,0){2}{ \multiput(0,0)(20,20){1}{\vector(-1, -1){5}}}
  \multiput( 0.2,39.6)(40,0){2}{ \multiput(0,0)(20,20){1}{\vector(-1, -1){5}}}
  \multiput( 0.2,39.6)(40,0){2}{ \multiput(0,0)(20,20){2}{\vector( 1, 2){5}}}
  \multiput( 0.5,39.4)(40,0){2}{ \multiput(0,0)(20,20){2}{\vector( 1, 2){5}}}
  \multiput( 0.8,39.1)(40,0){2}{ \multiput(0,0)(20,20){2}{\vector( 1, 2){5}}}
  \end{picture} } 
   \end{picture}
   \caption{\label{f-dof3} The (visible) degrees of freedom of the $C^1$-$Q_5$ Bell finite element
      on the top of cube, cf. \eqref{dof3}. All mid-edge $\partial_{xz}v$ and $\partial_{yz}v$
       are not visible.
         }
   \end{figure}
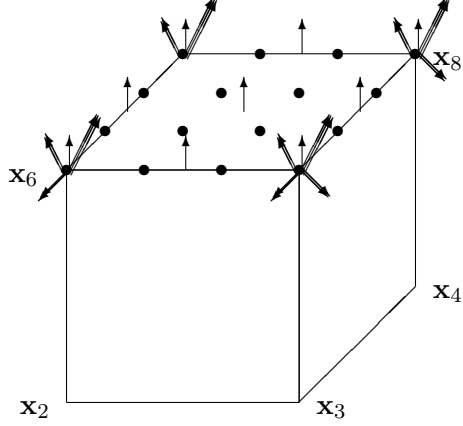

\begin{lemma}  
The set of degrees of freedom in \eqref{dof3} uniquely define a $V_T$ function in 
  \eqref{V-T}.
\end{lemma}

\begin{proof} The number of degrees of freedom in \eqref{dof3} is
\an{ \label{n3dof} \ad{ N_d &=64+12 (4k-15)+6 ((k-4)^2+(k-3)^2)+(k-3)^3 \\
                       &=k^3 + 3 k^2 - 9 k + 7. }  }
To reduce a 2D $Q_k$ polynomial's degree by 1,  we need to post $2k-1$ constraints for
  $x^k, x^{k} y, \dots, x^k y^k, \dots, x y^{k}$ and $ y^k$.
Thus the number of constraint equations in \eqref{V-T} is,  
\an{\label{n3e} N_e &=6(2k-1). }
Adding \eqref{n3dof} and \eqref{n3e},  we get
\a{ N_d+N_e &=k^3 +3 k^2 + 3 k +1 \\
            & =(k+1)^3=\dim Q_k.  }
Thus, to determine a $v_h\in V_T$ in 3D,  
   we have a square system of linear equations of $(k+1)^3$ unknowns.
The uniqueness implies the existence of the solution.

Let $v_h\in V_T$ and $F_m(v_h)=0$ for all dofs in \eqref{dof3}.   We need to show $v_h=0$.
It would use the proof in 2D several times.

On the top square $\b x_5\b x_6\b x_7\b x_8$, the restriction of $v_h$ is a 2D Bell $C^1$-$Q_k$ element.
By the degrees of freedom $\{ v(\b n_l^{(0)}), \partial_x v(\b n_l^{(1)}),   
   \partial_y v(\b n_l^{(2)}),  \partial_{xy} v(\b n_l^{(4)})\}$ and Lemma \ref{l2d},
   we have $v_h|_{\b x_5\b x_6\b x_7\b x_8} = 0 $ and 
\a{ v_h = (z-(\b x_5)_3) q_{k,k,k-1} }
for some $q_{k,k,k-1} \in \t{span} \{ x^i y^j z^l,
   \ 0\le i,j\le k, \ 0\le l\le k-1\}$.

On the top square $\b x_5\b x_6\b x_7\b x_8$, the restriction of $\partial_z v_h$ is a 2D 
  BFS $C^1$-$Q_{k-1}$ element because of the normal derivative restriction in \eqref{V-T}.
By the degrees of freedom $\{ \partial_{z}v(\b n_l^{(3)}), \partial_{xz} v(\b n_l^{(5)}),   
   \partial_{yz} v(\b n_l^{(6)}),  \partial_{xyz} v(\b n_l^{(7)})\}$ and 
    the uni-solvency of the 2D BFS element,
   we have $\partial_{z}v_h|_{\b x_5\b x_6\b x_7\b x_8} = 0 $ and 
\a{ v_h =(z-(\b x_5)_3)^2 q_{k ,k ,k-2} }
for some $q_{k ,k ,k-2} \in \t{span} \{ x^i y^j z^l,
   \ 0\le i,j\le k, \ 0\le l\le k-2\}$.
  
Repeating above analysis on the other five face squares,  we obtain 
\an{\label{qk-4} v_h = b^2 q_{k-4} }
for some $q_{k-4} \in Q_{k-4}$, where $b=(x-x_1)( x_1+h-x)(y-y_1)( y_1+h-y)
        (z-z_1)( z_1+h-z)$ and $\b x_1 = (x_1,y_1,z_1)$.

For the $Q_{k-4}$ polynomial $q_{k-4}$ in \eqref{qk-4},
  it and $v_h$ vanish at $(k-3)^3$ internal points, defined in \eqref{dof3}.
Thus $q_{k-4}=0$ and $v_h=0$.
  The lemma is proved. 
\end{proof}

\section{The finite element solution and convergence}

The $C^1$-$Q_k$ ($k\ge 4$) Bell finite element spaces are defined by
\an{ \label{V-h} V_h =\{ v_h \in H^2_0(\Omega) : v_h = \sum_{i=1}^N c_i \phi_i \}, }
where $\{ \phi_i\} $ is a basis of the global space which combines the local basis functions
  dual to the degrees of freedom \eqref{dof2} or \eqref{dof3}.

 The finite element discretization of the biharmonic equation \eqref{bi} reads:
   Find $u\in V_h $ such that
\an{\label{finite} (\Delta u, \Delta v) = (f, v) \quad \forall v\in V_h , }
where $V_h$ is defined in \eqref{V-h}.

\begin{lemma} The finite element equation 
    \eqref{finite} has a unique solution. \end{lemma}

\begin{proof}
By letting $v_h=u_h$ and $f=0$ in the linear equation 
    \eqref{finite},  we get $\Delta u_h=0$ on every $T$.
Because $u_h\in H^2_0(\Omega)$,   $u_h$ is a global solution of the Laplace equation,
\a{ \Delta u=0 \ \t{in }\ \Omega;  \quad u= 0 \ \t{on } \ \partial \Omega. }
By the uniqueness of the solution of the Laplace equation, $u_h=0$.  
Thus,  \eqref{finite} has a unique solution. 
 \end{proof}
                     
  In order to have the optimal order $L^2$ convergence,  we need to assume the
    domain $\Omega$ in \eqref{bi} is regular enough,  for example, a square or a 
      cube,  such that 
\an{\label{regular}           |u|_4 \le C \|f\|_0 .   }

\begin{theorem}  Let $ u\in H^{k+1}\cap H^2_0(\Omega)$ be
    the exact solution of the biharmoic equation \eqref{bi}.  
   Let $u_h$ be the $C^1$-$P_k$ finite element solution of \eqref{finite}.   
   Assuming the full-regularity  \eqref{regular}, it holds 
  \a{  \| u- u_h\|_{0} +  h^2  |  u- u_h |_{2}  
         & \le Ch^{k+1} | u|_{k+1}, \quad k\ge 4.  } 
\end{theorem}
                        
\begin{proof} The proof is standard.  A proof can be found in \cite{Zhang-bubble-p4}.

         Subtracting \eqref{finite}  from \eqref{bi},  we get
\a{ (\Delta( u- u_h), \Delta v_h)=0\quad \forall v_h\in   V_{h }. }
In $H^2_0(\Omega)$, the $H^2$ semi-norm $|u|_2$ is equivalent to the norm $\|\Delta u\|_0$.
For simplicity, we denote $\|\Delta u\|_0$ by $|u|_2$, i.e., omitting the 
  norm equivalence constants.
Applying the Schwarz inequality,  it follows that
\a{    |   u- u_h|_{2}^2  
     & =  (\Delta(  u-  u_h), \Delta(  u- I_h u))\\ 
     &\le |   u- u_h|_{2} |   u- I_h u |_{2} \le Ch^{k-1} |u|_{k+1}
 |   u- u_h|_{2} ,} 
      where $ I_h  u$ is the nodal interpolation defined by the degrees of freedom, i.e,
\a{  I_h u |_T =\sum_{n=1}^{N_T} F_n(u) \phi_n  }
  with $F_n$ and $\phi_n$   defined in \eqref{dof2} or \eqref{dof3}   and 
    their dual basis functions, respectively.
As $V_T\supset P_{k} (T)$ in \eqref{V-T},  
$I_h$ preserves $P_{k} $ polynomials locally.
We can use the spline theory in \cite{Lai} to prove the $H^2$-stability of $I_h$.
Or we can modify the theory of 
    \cite{Girault} to prove the $H^2$-stability of $I_h$.
Then It is routing to get
   the quasi-optimal approximation of the operator, cf. \cite{Scott-Zhang}.

To prove the $L^2$ convergence,  we introduce a dual problem,
  \an{ \label{d2}
    ( \Delta w, \Delta v) &=(u-u_h, v), 
     \quad w\in H^2_0(\Omega), \ \forall v \in H^2_0(\Omega) . }
Thus, by \eqref{d2}, 
\a{ \|u-u_h\|_0^2 &=(\Delta w, \Delta (u-u_h) ) = 
(\Delta (w-w_h), \Delta (u-u_h) ) \\
  & \le C h^2 |w|_4  h^{k-1} |u|_n \le C h^{k+1} |u|_{k+1}\|u-u_h\|_0, }
where $w_h$ is the finite element solution to the equation \eqref{d2}.
We obtain the $L^2$ error bound. The theorem is proved.
\end{proof}

\section{Numerical Experiments}

In the 2D numerical computation,  we solve the biharmonic equation \eqref{bi}
   on the unit square domain $\Omega=(0,1)\times(0,1)$. 
We choose an $f$ in \eqref{bi} so that the exact solution is
\an{\label{s2}
   u = \sin^2(\pi x)  \sin^2(\pi y) .  }  
 We compute the solution \eqref{s2} on the square grids shown in Figure \ref{f-2}, by 
  the newly constructed $C^1$-$Q_k$ Bell finite elements \eqref{V-h},
    and by the standard $C^1$-$Q_k$ BFS elements.
The results are listed in Tables \ref{t1}-\ref{t4}, 
    where we can see that the optimal orders of convergence 
  are achieved in all cases.  
Though both methods have the same order of convergence,
  we can compare further the constants before the order term $O(h^k)$.
For a low polynomial degree $k$,  the new element has roughly $3/4$ of unknowns and 
   twice the errors of the corresponding $C^1$-$Q_k$ BFS element.
See an example in counting the number of unknowns in Figure \ref{f-2e}.
Though the Bell elements take slightly less time by the conjugate gradient solver,
   it seems to be outperformed by the BFS elements for low $k$.
For a high polynomial degree $k$, the new element has roughly $9/10$ of unknowns and 
   about the same errors of the corresponding $C^1$-$Q_k$ BFS element.
The two methods perform about equally well.
   
\begin{figure}[H]
\begin{center}\setlength\unitlength{2pt}\centering 
 \begin{picture}(140,45)(0,0) \put(0,41){$G_1:$}  \put(50,41){$G_2:$} \put(100,41){$G_3:$} 
  
\def\sq{\begin{picture}(40,40)(0,0)
  \multiput(0,0)(40,0){2}{\line(0,1){40}}\multiput(0,0)(0,40){2}{\line(1,0){40}} \end{picture} }
  
\put(0,0){\begin{picture}(40,40)(0,0)
  \multiput(0,0)(0,40){1}{\multiput(0,0)(40,0){1}{\sq}} 
  \end{picture} }
  
\put(50,0){\setlength\unitlength{1pt}\begin{picture}(40,40)(0,0)
  \multiput(0,0)(0,40){2}{\multiput(0,0)(40,0){2}{\sq}} 
  \end{picture} } 
\put(100,0){\setlength\unitlength{0.5pt}\begin{picture}(40,40)(0,0)
  \multiput(0,0)(0,40){4}{\multiput(0,0)(40,0){4}{\sq}} 
  \end{picture} } 
\end{picture}\end{center}
\caption{The first three rectangular  grids for computing  \eqref{s2} in Tables \ref{t1}--\ref{t4}. }
\label{f-2}
\end{figure}
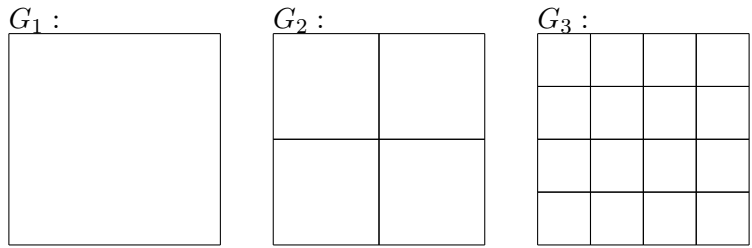

\begin{table}[H]
  \centering  \renewcommand{\arraystretch}{1.1}
  \caption{Error profile on the square meshes shown as in Figure \ref{f-2}, 
     for computing \eqref{s2}. }
  \label{t1}
\begin{tabular}{c|cc|cc|r}
\hline
$G_i$ &  $\| u-u_h\|_{0}$ & $O(h^r)$ & $|u-u_h|_{2}$& $O(h^r)$ & $\dim V_h$  \\ \hline
    &  \multicolumn{5}{c}{ By the new $C^1$-$Q_4$ Bell element \eqref{V-h}. }   \\
\hline  
 1&    0.273E+00 &  0.0&    0.140E+02 &  0.0 &      21\\
 2&    0.537E-02 &  5.7&    0.157E+01 &  3.1 &      52\\
 3&    0.327E-03 &  4.0&    0.356E+00 &  2.1 &     156\\
 4&    0.101E-04 &  5.0&    0.476E-01 &  2.9 &     532\\
 5&    0.306E-06 &  5.1&    0.598E-02 &  3.0 &    1956\\
 6&    0.928E-08 &  5.0&    0.742E-03 &  3.0 &    7492\\\hline 
    &  \multicolumn{5}{c}{ By the standard  $C^1$-$Q_4$ BFS element. }   \\
\hline   
 1&    0.151E+00 &  0.0&    0.113E+02 &  0.0 &      25\\
 2&    0.501E-02 &  4.9&    0.138E+01 &  3.0 &      64\\
 3&    0.165E-03 &  4.9&    0.171E+00 &  3.0 &     196\\
 4&    0.527E-05 &  5.0&    0.214E-01 &  3.0 &     676\\
 5&    0.166E-06 &  5.0&    0.267E-02 &  3.0 &    2500\\
 6&    0.518E-08 &  5.0&    0.334E-03 &  3.0 &    9604\\ 
\hline 
    \end{tabular}%
\end{table}%

\begin{table}[H]
  \centering  \renewcommand{\arraystretch}{1.1}
  \caption{Error profile on the square meshes shown as in Figure \ref{f-2}, 
     for computing \eqref{s2}. }
  \label{t2}
\begin{tabular}{c|cc|cc|r}
\hline
$G_i$ &  $\| u-u_h\|_{0}$ & $O(h^r)$ & $|u-u_h|_{2}$& $O(h^r)$ & $\dim V_h$  \\ \hline
    &  \multicolumn{5}{c}{ By the new $C^1$-$Q_5$ Bell element \eqref{V-h}. }   \\
\hline  
 1&    0.325E-01 &  0.0&    0.423E+01 &  0.0 &      32\\
 2&    0.117E-03 &  8.1&    0.500E-01 &  6.4 &      88\\
 3&    0.802E-05 &  3.9&    0.158E-01 &  1.7 &     284\\
 4&    0.130E-06 &  5.9&    0.958E-03 &  4.0 &    1012\\
 5&    0.206E-08 &  6.0&    0.595E-04 &  4.0 &    3812\\ \hline 
    &  \multicolumn{5}{c}{ By the standard  $C^1$-$Q_5$ BFS element. }   \\
\hline   
 1&    0.943E-02 &  0.0&    0.131E+01 &  0.0 &      36\\
 2&    0.155E-03 &  5.9&    0.762E-01 &  4.1 &     100\\
 3&    0.250E-05 &  6.0&    0.469E-02 &  4.0 &     324\\
 4&    0.395E-07 &  6.0&    0.292E-03 &  4.0 &    1156\\
 5&    0.619E-09 &  6.0&    0.182E-04 &  4.0 &    4356\\
\hline 
    \end{tabular}%
\end{table}%

\begin{table}[H]
  \centering  \renewcommand{\arraystretch}{1.1}
  \caption{Error profile on the square meshes shown as in Figure \ref{f-2}, 
     for computing \eqref{s2}. }
  \label{t3}
\begin{tabular}{c|cc|cc|r}
\hline
$G_i$ &  $\| u-u_h\|_{0}$ & $O(h^r)$ & $|u-u_h|_{2}$& $O(h^r)$ & $\dim V_h$  \\ \hline
    &  \multicolumn{5}{c}{ By the new $C^1$-$Q_6$ Bell element \eqref{V-h}. }   \\
\hline  
 1&    0.110E-02 &  0.0&    0.381E+00 &  0.0 &      45\\
 2&    0.707E-04 &  4.0&    0.496E-01 &  2.9 &     132\\
 3&    0.399E-06 &  7.5&    0.122E-02 &  5.3 &     444\\
 4&    0.319E-08 &  7.0&    0.398E-04 &  4.9 &    1620\\
 5&    0.266E-10 &  6.9&    0.126E-05 &  5.0 &    6180\\ \hline 
    &  \multicolumn{5}{c}{ By the standard  $C^1$-$Q_6$ BFS element. }   \\
\hline   
 1&    0.110E-02 &  0.0&    0.381E+00 &  0.0 &      49\\
 2&    0.707E-04 &  4.0&    0.496E-01 &  2.9 &     144\\
 3&    0.395E-06 &  7.5&    0.114E-02 &  5.4 &     484\\
 4&    0.310E-08 &  7.0&    0.359E-04 &  5.0 &    1764\\
 5&    0.258E-10 &  6.9&    0.112E-05 &  5.0 &    6724\\
\hline 
    \end{tabular}%
\end{table}%

\begin{table}[H]
  \centering  \renewcommand{\arraystretch}{1.1}
  \caption{Error profile on the square meshes shown as in Figure \ref{f-2}, 
     for computing \eqref{s2}. }
  \label{t4}
\begin{tabular}{c|cc|cc|r}
\hline
$G_i$ &  $\| u-u_h\|_{0}$ & $O(h^r)$ & $|u-u_h|_{2}$& $O(h^r)$ & $\dim V_h$  \\ \hline
    &  \multicolumn{5}{c}{ By the new $C^1$-$Q_7$ Bell element \eqref{V-h}. }   \\
\hline  
 1&    0.115E-02 &  0.0&    0.369E+00 &  0.0 &      60\\
 2&    0.681E-06 & 10.7&    0.743E-03 &  9.0 &     184\\
 3&    0.184E-07 &  5.2&    0.788E-04 &  3.2 &     636\\
 4&    0.733E-10 &  8.0&    0.122E-05 &  6.0 &    2356\\
 \hline 
    &  \multicolumn{5}{c}{ By the standard  $C^1$-$Q_7$ BFS element. }   \\
\hline   
 1&    0.115E-02 &  0.0&    0.369E+00 &  0.0 &      64\\
 2&    0.681E-06 & 10.7&    0.743E-03 &  9.0 &     196\\
 3&    0.183E-07 &  5.2&    0.750E-04 &  3.3 &     676\\
 4&    0.731E-10 &  8.0&    0.118E-05 &  6.0 &    2500\\
\hline 
    \end{tabular}%
\end{table}%

In the 3D numerical computation,  we solve the biharmonic equation \eqref{bi}
   on the unit cube domain $\Omega=(0,1)^3$. 
We choose an $f$ in \eqref{bi} so that the exact solution is
\an{\label{s3}
   u = \sin^2(\pi x)  \sin^2(\pi y) \sin^2(\pi z) .  }  
 We compute the solution \eqref{s3} on the cuboid grids shown in Figure \ref{f-3}, by 
  the newly constructed $C^1$-$Q_4$ Bell finite element \eqref{V-h},
    and by the standard $C^1$-$Q_4$ BFS element.

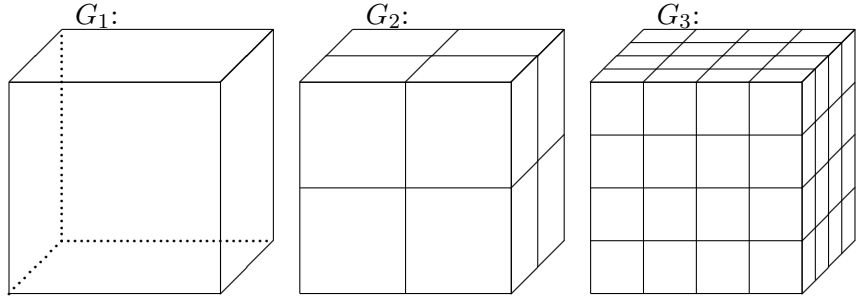
\begin{figure}[H] 
\begin{center}
 \setlength\unitlength{1pt}
    \begin{picture}(320,118)(0,3)
    \put(0,0){\begin{picture}(110,110)(0,0) \put(25,102){$G_1$:}
       \multiput(0,0)(80,0){2}{\line(0,1){80}}  \multiput(0,0)(0,80){2}{\line(1,0){80}}
       \multiput(0,80)(80,0){2}{\line(1,1){20}} \multiput(0,80)(20,20){2}{\line(1,0){80}}
       \multiput(80,0)(0,80){2}{\line(1,1){20}}  \multiput(80,0)(20,20){2}{\line(0,1){80}}
%    \put(80,0){\line(-1,1){80}}\put(80,0){\line(1,5){20}}\put(80,80){\line(-3,1){60}}
       \multiput(0,0)(2,2){10}{\circle*{1}}   \multiput(20,20)(2.5,0){32}{\circle*{1}}  
        \multiput(20,20)(0,2.5){32}{\circle*{1}}  
      \end{picture}}
    \put(110,0){\begin{picture}(110,110)(0,0)\put(25,102){$G_2$:}
       \multiput(0,0)(40,0){3}{\line(0,1){80}}  \multiput(0,0)(0,40){3}{\line(1,0){80}}
       \multiput(0,80)(40,0){3}{\line(1,1){20}} \multiput(0,80)(10,10){3}{\line(1,0){80}}
       \multiput(80,0)(0,40){3}{\line(1,1){20}}  \multiput(80,0)(10,10){3}{\line(0,1){80}}
%    \put(80,0){\line(-1,1){80}}\put(80,0){\line(1,5){20}}\put(80,80){\line(-3,1){60}}
%       \multiput(40,0)(40,40){2}{\line(-1,1){40}} 
%        \multiput(80,40)(10,-30){2}{\line(1,5){10}}
%        \multiput(40,80)(50,10){2}{\line(-3,1){30}}
      \end{picture}}
    \put(220,0){\begin{picture}(110,110)(0,0) \put(25,102){$G_3$:}
       \multiput(0,0)(20,0){5}{\line(0,1){80}}  \multiput(0,0)(0,20){5}{\line(1,0){80}}
       \multiput(0,80)(20,0){5}{\line(1,1){20}} \multiput(0,80)(5,5){5}{\line(1,0){80}}
       \multiput(80,0)(0,20){5}{\line(1,1){20}}  \multiput(80,0)(5,5){5}{\line(0,1){80}}
%    \put(80,0){\line(-1,1){80}}\put(80,0){\line(1,5){20}}\put(80,80){\line(-3,1){60}}
%       \multiput(40,0)(40,40){2}{\line(-1,1){40}} 
%        \multiput(80,40)(10,-30){2}{\line(1,5){10}}
%        \multiput(40,80)(50,10){2}{\line(-3,1){30}}
%
%       \multiput(20,0)(60,60){2}{\line(-1,1){20}}   \multiput(60,0)(20,20){2}{\line(-1,1){60}} 
%        \multiput(80,60)(15,-45){2}{\line(1,5){5}} \multiput(80,20)(5,-15){2}{\line(1,5){15}}
%        \multiput(20,80)(75,15){2}{\line(-3,1){15}}\multiput(60,80)(25,5){2}{\line(-3,1){45}}
      \end{picture}}

    \end{picture} 
    \end{center} 
\caption{ The first three  tetrahedral grids for the computation 
    in Table  \ref{t5}.  } 
\label{f-3}
\end{figure}

\begin{table}[H]
  \centering  \renewcommand{\arraystretch}{1.1}
  \caption{Error profile on the cuboid meshes shown as in Figure \ref{f-3}, 
     for computing \eqref{s3}. }
  \label{t5}
\begin{tabular}{c|cc|cc|r}
\hline
$G_i$ &  $\| u-u_h\|_{0}$ & $O(h^r)$ & $|u-u_h|_{2}$& $O(h^r)$ & $\dim V_h$  \\ \hline
    &  \multicolumn{5}{c}{ By the new $C^1$-$Q_4$ 3D Bell element \eqref{V-h}. }   \\
\hline  
 1 &   0.710E-01 &0.00 &   0.260E+01 &0.00 &      89 \\
 2 &   0.671E-02 &3.41 &   0.126E+01 &1.04 &     350 \\
 3 &   0.125E-03 &5.75 &   0.142E+00 &3.15 &    1844 \\
 4 &   0.450E-05 &4.80 &   0.217E-01 &2.72 &   11744 \\
 5 &   0.153E-06 &4.88 &   0.294E-02 &2.88 &   83384 \\
 \hline 
    &  \multicolumn{5}{c}{ By the standard  $C^1$-$Q_7$ 3D BFS element. }   \\
\hline    
 1 &   0.710E-01 &0.00 &   0.260E+01 &0.00 &     125 \\
 2 &   0.671E-02 &3.41 &   0.126E+01 &1.04 &     512 \\
 3 &   0.109E-03 &5.94 &   0.111E+00 &3.51 &    2744 \\
 4 &   0.342E-05 &4.99 &   0.139E-01 &3.00 &   17576 \\
 5 &   0.107E-06 &5.00 &   0.173E-02 &3.00 &  125000 \\ 
\hline 
    \end{tabular}%
\end{table}%

\section*{Ethical Statement}

\subsection*{Compliance with Ethical Standards} { \ } 
   The submitted work is original and is not published elsewhere in any form or language.

This article does not contain any studies involving animals.
This article does not contain any studies involving human participants.

\subsection*{Availability of supporting data}  
   Data sharing is not applicable to this article since no datasets were generated or collected 
 in the work.

\subsection*{Competing interests} 
All authors declare that they have no potential conflict of interest.

\subsection*{Authors' contributions}
All authors made equal contribution.

\end{document}